\documentclass[a4paper,12pt]{amsart}
\usepackage[utf8]{inputenc}
\usepackage[T1]{fontenc}
\usepackage[UKenglish]{babel}
\usepackage[a4paper,margin=28mm]{geometry}
\usepackage{verbatim}
\usepackage{caption}
\usepackage[pdftex]{graphicx}
\usepackage{epstopdf}
\allowdisplaybreaks[4]
\usepackage{times}
\usepackage{dsfont,mathrsfs}
\usepackage{amsmath}
\usepackage{amsthm}
\usepackage{amssymb}
\usepackage{amsfonts}
\usepackage{latexsym}
\usepackage{booktabs}
\usepackage{graphicx}

\usepackage[colorlinks,pagebackref]{hyperref}
\usepackage{xcolor}

\newtheorem{theorem}{Theorem}[section]
\newtheorem{lemma}[theorem]{Lemma}

\newtheorem{corollary}[theorem]{Corollary}

\theoremstyle{definition}
\newtheorem{definition}[theorem]{Definition}

\newtheorem{remark}[theorem]{Remark}
\numberwithin{equation}{section}
\renewcommand{\labelenumi}{\roman{enumi})}
\renewcommand\theenumi\labelenumi
\renewcommand{\leq}{\leqslant}
\renewcommand{\le}{\leqslant}
\renewcommand{\geq}{\geqslant}
\renewcommand{\ge}{\geqslant}

\newcommand{\Be}{\begin{equation}}
\newcommand{\Ees}{\end{equation*}}
\newcommand{\Bes}{\begin{equation*}}
\newcommand{\Ee}{\end{equation}}

\newcommand{\R}{\mathbb{R}}

\newcommand{\e}{\varepsilon}
\newcommand{\PP}{\mathbb{P}}

\usepackage{marginnote}
\begin{document}

\title[A generalized Catoni's ${\rm M}$-estimator under finite $\alpha$-th moment assumption]
{A generalized Catoni's ${\rm M}$-estimator under finite {$\alpha$-th moment assumption} with $\alpha \in (1,2)$}
\author[P. Chen]{Peng Chen}
\address{Peng Chen: 1. Department of Mathematics,
Faculty of Science and Technology,
University of Macau,
Av. Padre Tom\'{a}s Pereira, Taipa
Macau, China; \ \ 2. UM Zhuhai Research Institute, Zhuhai, China.}
\email{yb77430@um.edu.mo}

\author[X. Jin]{Xinghu Jin}
\address{Xinghu Jin: 1. Department of Mathematics,
Faculty of Science and Technology,
University of Macau,
Av. Padre Tom\'{a}s Pereira, Taipa
Macau, China; \ \ 2. UM Zhuhai Research Institute, Zhuhai, China.}
\email{yb77438@um.edu.mo}

\author[X. Li]{Xiang Li}
\address{Xiang Li: 1. Department of Mathematics,
Faculty of Science and Technology,
University of Macau,
Av. Padre Tom\'{a}s Pereira, Taipa
Macau, China; \ \ 2. UM Zhuhai Research Institute, Zhuhai, China.}
\email{yc07904@um.edu.mo}

\author[L. Xu]{Lihu Xu}
\address{Lihu Xu: 1. Department of Mathematics,
Faculty of Science and Technology,
University of Macau,
Av. Padre Tom\'{a}s Pereira, Taipa
Macau, China; \ \ 2. UM Zhuhai Research Institute, Zhuhai, China.}
\email{lihuxu@um.edu.mo}
\begin{abstract}
We generalize the { ${\rm M}$-estimator} put forward by Catoni in his seminal paper \cite{C12} to the case in which samples can have finite $\alpha$-th moment with $\alpha \in (1,2)$ rather than finite variance, our approach is by slightly modifying the influence function $\varphi$ therein. The choice of the new influence function is inspired by the Taylor-like expansion developed in \cite{C-N-X}.  We obtain a deviation bound of the estimator, as $\alpha \rightarrow 2$, this bound is the same as that in \cite{C12}. Experiment shows that our generalized ${\rm M}$-estimator performs better than the empirical mean estimator, the smaller the $\alpha$ is, the better the performance will be. As an application, we study an $\ell_{1}$ regression considered by Zhang et al. \cite{Z-Z} who assumed that samples have finite variance, and relax their assumption to be finite {$\alpha$-th} moment with $\alpha \in (1,2)$.
\end{abstract}

\maketitle

\section{Introduction}
Let $X_{1},...,X_{n}$ be a {sequence of samples} drawn from a distribution, its empirical mean estimator is defined by
$$\bar X=\frac{X_{1}+...+X_{n}}n.$$
$\bar X$ has an
optimal minimax mean square error among all mean estimators in all models including Gaussian distributions, but its deviation is suboptimal for heavy tail distribution \cite{C12}.

Catoni put forward in his seminal paper \cite{C12} a new ${\rm M}$-estimator for {heavy-tailed} samples with finite variance by solving the following equation about $\theta$
\begin{align*}
\sum_{i=1}^{n}\varphi\big(\beta(X_{i}-{\theta})\big)=0 \ \quad {\rm with} \quad  \ -\log\big(1-x+\frac{|x|^{2}}{2}\big)\leq\varphi(x)\leq\log\big(1+x+\frac{|x|^{2}}{2}\big),
\end{align*}
where $\beta>0$ is a parameter to be tuned and $\varphi$ is non-decreasing and called influence function. The deviation performance of this estimator is much better than $\bar X$. Catoni's idea has been broadly applied to many research problems, see for instance \cite{B-C,W-M,F-W,H18,H,L-M,S-Z-F,X-Z}. The finite variance assumption plays an important role in Catoni's analysis, but it rules out many interesting distributions such as Pareto law \cite{K17,X,C-N-X,J-L-L}, which describes the distributions of wealth and social networks.

We generalize Catoni's ${\rm M}$-estimator to the case in which samples can have finite $\alpha$-th moment with $\alpha \in (1,2)$. Our approach is by replacing Catoni's influence function with the one satisfying $-\log\big(1-x+\frac{|x|^{\alpha}}{\alpha}\big)\leq\varphi(x)\leq\log\big(1+x+\frac{|x|^{\alpha}}{\alpha}\big)$.  The choice of the new influence function is inspired by the Taylor-like expansion developed in \cite{C-N-X}. By an argument very similar to Catnoi's, we obtain a deviation upper bound which coincides with that in \cite{C12} as $\alpha \uparrow 2$ (see Theorem \ref{main1} and Remark \ref{r:MThm} below). Experiment shows that our generalized ${\rm M}$-estimator performs better than the empirical mean estimator, the smaller the $\alpha$ is, the better the performance will be.

Catoni's argument for finding ${\rm M}$-estimator can be split into two steps. The first step is to find two deterministic values $\theta_{-}$ and $\theta_{+}$, both depending on a parameter $\beta$ to be tuned later, such that the ${\rm M}$-estimator $\hat \theta$ falls between $\theta_{-}$ and $\theta_{+}$ with a high probability.  The $\theta_{-}$ and $\theta_{+}$ were obtained explicitly by solving two quadratic algebraic equations $B_{-}(\theta)=0$ and $B_{+}(\theta)=0$ respectively in \cite{C12}, whereas in our setting the corresponding equations are not quadratic and the solutions do not have explicit forms. Alternatively, we first prove that $B_{-}(\theta)=0$ has a largest solution, while $B_{+}(\theta)=0$ has a smallest one, and then use them as a replacement of $\theta_{-}$ and $\theta_{+}$ in our analysis.  The second step is to show that as one chooses $\beta>0$ sufficiently small, the difference between $\theta_{-}$ and $\theta_{+}$ can be as small as we wish,  whence the estimator can be localized in a small interval with a high probability. In our analysis, we also {choose} a sufficiently small $\beta$ (depending on $\alpha$) to make our estimator fall in a small interval with the above special solutions as its two end points. As $\alpha \uparrow 2$, our result coincides with that in \cite{C12}.
\vskip 3mm

As an application of our generalized estimator, we consider the $\ell_{1}$- regression with heavy-tailed samples studied by Zhang et al. \cite{Z-Z} who assumed the samples have finite variance. The linear regression considered in \cite{Z-Z} aims to find the minimizer $\theta^{*}$ of the minimization problem as follows
\begin{align*}
\min_{\mathbf{\theta}\in\mathbf{\Theta}} R_{\ell_{1}}(\theta) \ \ \ \ {\rm with} \ \ \ R_{\ell_{1}}(\theta)=\mathbb{E}_{(\mathbf{x},y)\sim{\bf\Pi}}\big[|\mathbf{x}^{T}\mathbf{\theta}-y|\big],
\end{align*}
where ${\bf\Pi}$ is a probability distribution, and ${ \mathbf{\Theta}\subseteq\mathbb{R}^{d}}$ is the set in which $\theta^{*}$ is located. In practice, $\mathbf{\Pi}$ is not known, one usually draws a data set $\mathcal{T}=(\mathbf{x}_{1},y_{1}),\cdots,(\mathbf{x}_{n},y_{n})$ from $\mathbf{\Pi}$ and considers the following empirical optimization problem:
\begin{align*}
\min_{\mathbf{\theta}\in\mathbf{\Theta}}\widehat{R}_{\ell_{1}}(\mathbf{\theta}) \ \ \ {\rm with} \ \ \ \widehat{R}_{\ell_{1}}(\mathbf{\theta})=\frac{1}{n}\sum_{i=1}^{n}|\mathbf{x}_{i}^{T}\mathbf{\theta}-y_{i}|.
\end{align*}
The theoretical guarantees for bounded or sub-Gaussian distributed ${\bf\Pi}$ have been discussed in many papers, see for instance \cite{B-M,K11,Z-Y-J}.

Inspired by Catoni's work, Zhang et al. considered the case that ${\bf\Pi}$ is {heavy-tailed} with finite variance and proposed a new minimization problem
$$\min_{\mathbf{\theta}\in\mathbf{\Theta}} \widehat R_{\varphi, \ell_{1}}(\theta) \ \ \ {\rm with} \ \ \ \widehat R_{\varphi, \ell_{1}}(\theta)=\frac{1}{n\beta}\sum_{i=1}^{n}\varphi(\beta|y_{i}-\mathbf{x}_{i}^{T}\theta|),$$
where $\varphi$ is the same as that in \cite{C12} and $\beta>0$ is a parameter to be tuned. A new estimator was established from this minimization problem and an error bound was obtained. When the sample size $n$ tends to infinity, this error bound tends to zero.

Thanks to the analysis of Section \ref{deviation}, we extend the results in \cite{Z-Z} to the case in which samples can have finite $\alpha$-th moment with $\alpha \in (1,2)$, our approach is by replacing the original $\varphi$ with the one in Section \ref{deviation} and solving the corresponding minimization problem. We establish a similar error bound for our estimator and prove that it tends to zero as $n \rightarrow \infty$.
\vskip 3mm
The organization of this paper is as follows. In Section \ref{deviation}, we give the deviation analysis of ${\rm M}$-estimator for the case that only $\alpha$-th central moment is known and show that the ${\rm M}$-estimator has a better performance compared with the empirical mean. In Section \ref{empirical}, we state the upper bounds and the corresponding lower bounds on the empirical mean. In the last section, we discuss the $\ell_{1}$-regression with heavy-tailed distributions (variance equals infinity).

\section{A generalized Catoni's ${\rm M}$-estimator and its deviation analysis}\label{deviation}

Let $(X_{i})_{i=1}^{n}$ be a sequence of i.i.d. samples drawn from some unknown probability distribution $\mathbf{\Pi}$ on $\mathbb{R}$.
We assume that there exists some $\alpha \in (1,2)$ such that
{
$$\mathbb{E}|X_{1}|^{\alpha}<\infty.$$}
Further denote
{
\begin{align*}
m=\mathbb{E}[X_{1}], \quad \quad v=\mathbb{E}|X_{1}-m|^{\alpha}.
\end{align*}
}

Inspired from Catoni's idea in \cite{C12} and the Taylor-like expansion develop in \cite{C-N-X}, we consider a non-decreasing function $\varphi:\mathbb{R}\rightarrow\mathbb{R}$ such that
\begin{align}\label{trunc}
-\log\big(1-x+\frac{|x|^{\alpha}}{\alpha}\big)\leq\varphi(x)\leq\log\big(1+x+\frac{|x|^{\alpha}}{\alpha}\big).
\end{align}
The $widest$ possible choice of $\varphi$ (see Figure \ref{figure1}) compatible with these inequalities is
{
\begin{align*}
\varphi(x)=
\begin{cases}
\log\big(1+x+\frac{x^{\alpha}}{\alpha}\big), \quad &x\geq0,\\
-\log\big(1-x+\frac{|x|^{\alpha}}{\alpha}\big), \quad &x<0.
\end{cases}
\end{align*}
}
\begin{figure}
\centering
\includegraphics[height=8cm,width=14cm]{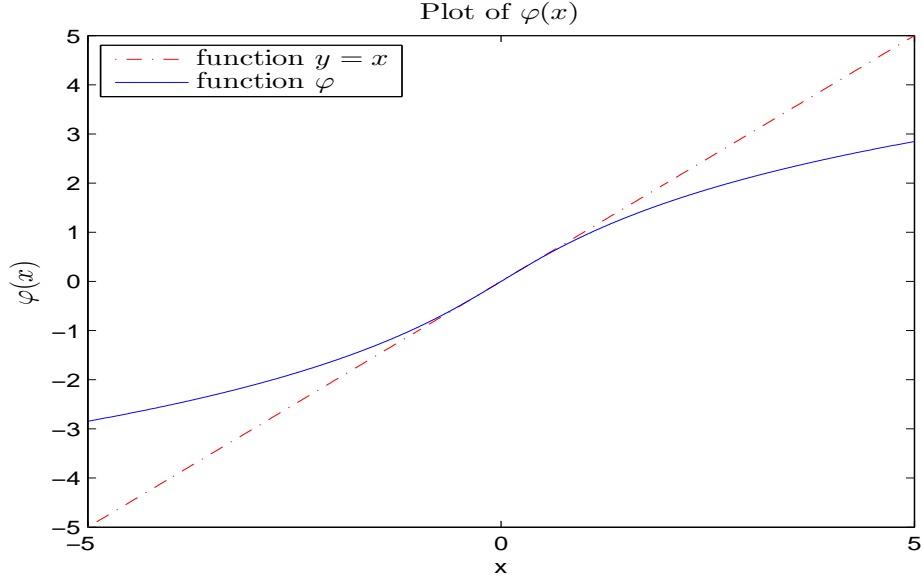}
\caption{$widest$ possible choice of $\varphi$}
\label{figure1}
\end{figure}

Let $\beta$ be some {strictly} positive real parameter that will be chosen later and denote the estimator of the mean $m$ by $\hat{\theta},$ which is the solution to the equation
\begin{align*}
\sum_{i=1}^{n}\varphi\big(\beta(X_{i}-\hat{\theta})\big)=0.
\end{align*}
For further use, we denote
\begin{align}\label{con}
r(\theta)=\frac{1}{\beta n}\sum_{i=1}^{n}\varphi\big(\beta(X_{i}-\theta)\big), \quad \theta\in\mathbb{R}.
\end{align}
It is easy to see $r(\theta)$ is a non-increasing random variable since $\varphi$ is non-decreasing.

Let us briefly explain how to find the estimator $\hat \theta$ into the following two steps: the first step is to find two deterministic values $\theta_{-}$ and $\theta_{+}$, both depending on $\beta$, such that $r(\theta_{-})>0>r(\theta_{+})$ with a high probability, from the non-decreasing property of $r$, we know that $\theta_{-}<\hat{\theta}<\theta_{+}$ holds with a high probability; the second is to show that as we choose $\beta>0$ sufficiently small, the difference between $\theta_{-}$ and $\theta_{+}$ can be as small as we wish,  whence the estimator can be localized in a small interval with a high probability.

\begin{lemma}
Keep the same notations and assumptions as above. Then, {for any $\theta \in \R$,} we have
\begin{align}\label{ineq1}
\mathbb{E}\big[\exp\big(\beta nr(\theta)\big)\big]\leq\exp\Big(n\beta(m-\theta)+\frac{2^{\alpha-1}n\beta^{\alpha}}{\alpha}\big(v+|m-\theta|^{\alpha}\big)\Big)
\end{align}
and
\begin{align}\label{ineq2}
\mathbb{E}\big[\exp\big(-\beta nr(\theta)\big)\big]
\leq \exp\Big(-n\beta(m-\theta)+\frac{2^{\alpha-1}n\beta^{\alpha}}{\alpha}\big(v+|m-\theta|^{\alpha}\big)\Big).
\end{align}
\end{lemma}
\begin{proof}
We first recall the ${\rm C}_{r}$-inequality, that is, for any $a,b\in\mathbb{R}$ and $p>0,$ we have
\begin{align*}
|a+b|^{p}\leq\max\{2^{p-1},1\}(|a|^{p}+|b|^{p}),
\end{align*}
 and if $p>1,$ it is easy to verify
\begin{align}\label{CR}
 |a+b|^{p}\leq2^{p-1}(|a|^{p}+|b|^{p}).
\end{align}
Then, noting that $X_{i},$ $i=1,\cdots,n$ are i.i.d., by (\ref{trunc}), we have
\begin{align*}
\mathbb{E}\big[\exp\big(\beta nr(\theta)\big)\big]=&\mathbb{E}\Big[\exp\big[\sum_{i=1}^{n}\varphi\big(\beta(X_{i}-\theta)\big)\big]\Big]\\
=&\Big(\mathbb{E}\Big[\exp\big[\varphi\big(\beta(X_{1}-\theta)\big)\big]\Big]\Big)^{n}\\
\leq&\Big(\mathbb{E}\Big[1+\beta(X_{1}-\theta)+\frac{\beta^{\alpha}}{\alpha}|X_{1}-\theta|^{\alpha}\Big]\Big)^{n},
\end{align*}
and noting that $\alpha\in(1,2),$ by (\ref{CR}), we further have
\begin{align*}
\mathbb{E}\Big[\exp\big(\beta nr(\theta)\big)\Big]\leq&\Big[1+\beta(m-\theta)+\frac{\beta^{\alpha}}{\alpha}\mathbb{E}|X_{1}-m+m-\theta|^{\alpha}\Big]^{n}\\
\leq&\Big[1+\beta(m-\theta)+\frac{2^{\alpha-1}\beta^{\alpha}}{\alpha}(v+|m-\theta|^{\alpha})\Big]^{n}\\
\leq&\exp\Big(n\beta(m-\theta)+\frac{2^{\alpha-1}n\beta^{\alpha}}{\alpha}\big(v+|m-\theta|^{\alpha}\big)\Big),
\end{align*}
where the last inequality is by the inequality $1+x\leq e^{x}$ for any $x\in\mathbb{R},$ (\ref{ineq1}) is proved and the inequality (\ref{ineq2}) can be proved in the same way. The proof is complete.
\end{proof}

According to (\ref{ineq1}) and (\ref{ineq2}), for any {$\epsilon\in(0,\frac{1}{2}),$} we denote
\begin{align*}
B_{+}(\theta)=m-\theta+\frac{(2\beta)^{\alpha-1}}{\alpha}(v+|m-\theta|^{\alpha})+\frac{\log(\epsilon^{-1})}{n\beta},
\end{align*}
\begin{align*}
B_{-}(\theta)=m-\theta-\frac{(2\beta)^{\alpha-1}}{\alpha}(v+|m-\theta|^{\alpha})-\frac{\log(\epsilon^{-1})}{n\beta}.
\end{align*}

\begin{lemma}\label{highprobability}
Keep the same notations and assumptions as above. Then, {for any $\theta \in \R$,} we have
\begin{align}\label{high1}
\mathbb{P}\big(r(\theta)<B_{+}(\theta)\big)\geq1-\epsilon
\end{align}
and
\begin{align}\label{high2}
\mathbb{P}\big(r(\theta)>B_{-}(\theta)\big)\geq1-\epsilon.
\end{align}
In particular, {for any $\theta \in \R$,} we have
\begin{align}\label{high00}
\mathbb{P}\big(B_{-}(\theta)<r(\theta)<B_{+}(\theta)\big)\geq1-2\epsilon.
\end{align}
\end{lemma}

\begin{proof}
By Markov inequality and (\ref{ineq1}), we have
\begin{align*}
\mathbb{P}\big(r(\theta)\geq B_{+}(\theta)\big)=&\mathbb{P}\Big(\exp\big(n\beta r(\theta)\big)\geq\exp\big(n\beta B_{+}(\theta)\big)\Big)\\
\leq&\frac{\mathbb{E}\big[\exp\big(n\beta r(\theta)\big)\big]}{\exp\Big(n\beta\big(m-\theta+\frac{(2\beta)^{\alpha-1}}{\alpha}(v+|m-\theta|^{\alpha})+\frac{\log(\epsilon^{-1})}{n\beta}\big)\Big)}\\
\leq&\frac{\exp\Big(n\beta(m-\theta)+\frac{2^{\alpha-1}n\beta^{\alpha}}{\alpha}\big(v+|m-\theta|^{\alpha}\big)\Big)}
{\exp\Big(n\beta(m-\theta)+\frac{2^{\alpha-1}n\beta^{\alpha}}{\alpha}(v+|m-\theta|^{\alpha})+\log(\epsilon^{-1})\big)\Big)}=\epsilon,
\end{align*}
(\ref{high1}) is proved. With the help of (\ref{ineq2}), (\ref{high2}) can be proved in the same way. The estimate \eqref{high00} immediately follows from \eqref{high1} and \eqref{high2}.
\end{proof}

Let us now consider the following two equations:
\begin{align}\label{equa+}
B_{+}(\theta)=0,
\end{align}
\begin{align}\label{equa-}
B_{-}(\theta)=0,
\end{align}
the following lemma tells us that they both have {at least one solution.}


\begin{lemma}\label{existence}
 Choosing the {strictly}  positive real parameter $\beta$ satisfying
 \begin{align}\label{exisineq}
\beta^{\alpha}v\leq\frac{\alpha-1}{2^{\alpha}}-\frac{\alpha\log(\epsilon^{-1})}{2^{\alpha-1}n}.
\end{align}
Then, both equations (\ref{equa+}) and (\ref{equa-}) have at least one solution. Furthermore, choosing the parameter $\beta$ satisfying
\begin{align}\label{exisineq2}
\frac{(2\beta)^{\alpha-1}}{\alpha}(v+1)+\frac{\log(\epsilon^{-1})}{n\beta}\leq1,
\end{align}
then $\theta_{+}\in(m,m+1]$ and $\theta_{-}\in[m-1,m),$ where $\theta_{+}$ is the smallest solution to the equation $B_{+}(\theta)=0$ and $\theta_{-}$ is the largest solution to the equation $B_{-}(\theta)=0$.
\end{lemma}
\begin{proof}
When $\theta\leq m,$ $B_{+}(\theta)>0.$ When $\theta>m,$ the derivative of function $B_{+}(\theta)$ is
\begin{align*}
B_{+}'(\theta)=-1+(2\beta)^{\alpha-1}(\theta-m)^{\alpha-1}.
\end{align*}
It is easy to verify that $B_{+}(\theta)$ is decreasing in $(m,m+\frac{1}{2\beta})$ and increasing in $[m+\frac{1}{2\beta},\infty),$ so $\theta=m+\frac{1}{2\beta}$ is the minimum point of function $B_{+}(\theta)$ in $(m,\infty).$ What's more, since $\lim_{\theta\rightarrow\infty}B_{+}(\theta)=\infty$ and by (\ref{exisineq}), we have
\begin{align*}
B_{+}(m+\frac{1}{2\beta})=&\frac{(2\beta)^{\alpha-1}}{\alpha}v+\frac{1}{2\alpha\beta}+\frac{\log(\epsilon^{-1})}{n\beta}-\frac{1}{2\beta}\\
=&\frac{2^{\alpha-1}}{\alpha\beta}\big(\beta^{\alpha}v-\frac{\alpha-1}{2^{\alpha}}+\frac{\alpha\log(\epsilon^{-1})}{2^{\alpha-1}n}\big)\leq0.
\end{align*}
When $B_{+}(m+\frac{1}{2\beta})=0,$ that is, $\beta^{\alpha}v=\frac{\alpha-1}{2^{\alpha}}-\frac{\alpha\log(\epsilon^{-1})}{2^{\alpha-1}n},$ then the equation $B_{+}(\theta)=0$ has the unique solution $\theta=m+\frac{1}{2\beta}.$ When $B_{+}(m+\frac{1}{2\beta})<0,$ that is, $\beta^{\alpha}v<\frac{\alpha-1}{2^{\alpha}}-\frac{\alpha\log(\epsilon^{-1})}{2^{\alpha-1}n},$ then the equation $B_{+}(\theta)=0$ has two solutions.

Furthermore, when $\theta\leq m,$
\begin{align*}
B_{+}(\theta)>0,
\end{align*}
which implies that {the smallest solution} $\theta_{+}>m.$ By (\ref{exisineq2}), we have
\begin{align*}
B_{+}(m+1)=-1+\frac{(2\beta)^{\alpha-1}}{\alpha}(v+1)+\frac{\log(\epsilon^{-1})}{n\beta}\leq0.
\end{align*}
Hence, there exists $\tilde{\theta}\in(m,m+1]$ such that $B_{+}(\tilde{\theta})=0.$ Since $\theta_{+}$ is the smallest solution to the equation $B_{+}(\theta)=0,$ we have $\theta_{+}\in(m,m+1].$

By the same argument as above, we can also prove the equation $B_{-}(\theta)=0$ has at least one solution and $\theta_{-}\in[m-1,m)$. The proof is complete.
\end{proof}
\begin{lemma}
Let $\beta, \theta_{-}$ and $\theta_{+}$ be specified in Lemma \ref{existence}. We have
\begin{align}\label{three}
\PP\left(\theta_{-} \le \hat{\theta} \le \theta_{+}\right) \ge {1-2 \epsilon}.
\end{align}
\end{lemma}
\begin{proof} By \eqref{high00}, we know that the following event holds with a probability {at least} $1-2\epsilon$: 
$$r(\theta_{-})>0 \ \ \ {\rm and} \ \ \ r(\theta_{+})<0.$$
Since $r(\theta)$ is a continuous function and non-increasing,  $r(\theta)=0$ has a solution $\hat \theta$ between $\theta_{-}$ and $\theta_{+}$ such that
$$\theta_{-} \le \hat \theta \le \theta_{+}$$
holds with a probability {at least} $1-2 \epsilon$. 
\end{proof}
\begin{remark}
From the above proof, it is easy to deduce (\ref{exisineq}) from (\ref{exisineq2}), that is, (\ref{exisineq2}) is a stronger condition.
\end{remark}

Now, we can give the main result in this section, which can give a deviation upper bound for the ${\rm M}$-estimator $\hat{\theta}$.

\begin{theorem}\label{main1}
For any {$\epsilon\in(0,\frac{1}{2})$} and positive integer $n,$ we assume
\begin{align}\label{assu}
n\geq(\frac{2v+1}{\alpha})^{\frac{\alpha}{\alpha-1}}\frac{2\alpha\log(\epsilon^{-1})}{v}.
\end{align}
Then, the inequality
\begin{align*}
|m-\hat{\theta}|< \frac{2\big(\frac{2\alpha\log(\epsilon^{-1})}{n}\big)^{\frac{\alpha-1}{\alpha}}v^{\frac{1}{\alpha}}}{\alpha-\big(\frac{2\alpha\log(\epsilon^{-1})}{nv}\big)^{\frac{\alpha-1}{\alpha}}}
=O\Big(\big(\log(\epsilon^{-1})\big)^{\frac{\alpha-1}{\alpha}}n^{-\frac{\alpha-1}{\alpha}}\Big)
\end{align*}
holds with probability {at least} $1-2\epsilon.$
\end{theorem}

\begin{remark} \label{r:MThm}
In Theorem \ref{main1}, if we take $\alpha=2,$ then the right hand side is $O\Big(\big(\log(\epsilon^{-1})\big)^{\frac{1}{2}}n^{-\frac{1}{2}}\Big),$ which has the same order as \cite[Proposition 2.4]{C12}.
\end{remark}

\begin{proof}
By (\ref{assu}),  (\ref{exisineq2}) is fulfilled and thus
 Lemma \ref{existence} tells us that the equation
\begin{align*}
B_{+}(\theta)=m-\theta+\frac{(2\beta)^{\alpha-1}}{\alpha}(v+|m-\theta|^{\alpha})+\frac{\log(\epsilon^{-1})}{n\beta}=0
\end{align*}
has the smallest solution $\theta_{+}\in(m,m+1].$ Denote $x=m-\theta,$ $x_{+}=m-\theta_{+}$ and
\begin{align*}
\tilde{B}_{+}(x)=x+\frac{(2\beta)^{\alpha-1}}{\alpha}(v+|x|^{\alpha})+\frac{\log(\epsilon^{-1})}{n\beta},
\end{align*}
then $\tilde{B}_{+}(x_{+})=0$ with $x_{+}\in[-1,0).$ Since $\alpha \in (1,2)$, we further have
\begin{align}\label{coup}
x_{+}+\frac{(2\beta)^{\alpha-1}}{\alpha}(v-x_{+})+\frac{\log(\epsilon^{-1})}{n\beta}\geq0.
\end{align}
By (\ref{exisineq2}), we have $1-\frac{(2\beta)^{\alpha-1}}{\alpha}>0,$ solving \eqref{coup} gives
\begin{align*}
x_{+}\geq-\frac{(2\beta)^{\alpha-1}v+\frac{\alpha\log(\epsilon^{-1})}{n\beta}}{\alpha-(2\beta)^{\alpha-1}}.
\end{align*}
Let $(2\beta)^{\alpha-1}v=\frac{\alpha\log(\epsilon^{-1})}{n\beta},$ i.e., $\beta=\frac{1}{2}\big(\frac{2\alpha\log(\epsilon^{-1})}{nv})^{\frac{1}{\alpha}}$, we have
\begin{align*}
m-\theta_{+}\geq-\frac{2\big(\frac{2\alpha\log(\epsilon^{-1})}{n}\big)^{\frac{\alpha-1}{\alpha}}v^{\frac{1}{\alpha}}}{\alpha-\big(\frac{2\alpha\log(\epsilon^{-1})}{nv}\big)^{\frac{\alpha-1}{\alpha}}}.
\end{align*}
Similarly, we get
\begin{align*}
m-\theta_{-}\leq \frac{2\big(\frac{2\alpha\log(\epsilon^{-1})}{n}\big)^{\frac{\alpha-1}{\alpha}}v^{\frac{1}{\alpha}}}{\alpha-\big(\frac{2\alpha\log(\epsilon^{-1})}{nv}\big)^{\frac{\alpha-1}{\alpha}}}.
\end{align*}
Then, (\ref{three}) implies that the inequality
\begin{align*}
|\hat{\theta}-m| \le \frac{2\big(\frac{2\alpha\log(\epsilon^{-1})}{n}\big)^{\frac{\alpha-1}{\alpha}}v^{\frac{1}{\alpha}}}{\alpha-\big(\frac{2\alpha\log(\epsilon^{-1})}{nv}\big)^{\frac{\alpha-1}{\alpha}}}
\end{align*}
holds with probability at least $1-2\epsilon$.
\end{proof}

The empirical mean estimator is defined by
\begin{align*}
{ \bar{X}=\frac{1}{n}\sum_{i=1}^{n}X_{i},}
\end{align*}
we postpone to study deviation bounds for the empirical mean $\bar{X}$ in Section \ref{empirical} below.

In Figures \ref{figure2}-\ref{figure5}, we compare the bound on the deviations of the ${\rm M}$-estimator $\hat{\theta}$ with the deviations of the empirical mean $\bar{X}$, when the sample distribution is {a Pareto distribution with shape parameter $\frac{2+\alpha}{2}$ and scale parameter $(\frac{2+\alpha}{2-\alpha})^{-\frac{1}{\alpha}}$ (see, e.g., \cite[Chapter 23]{K17}), that is,
\begin{align*}
\mathbb{P}(X_{1}\geq x)=2^{-1}(\frac{2+\alpha}{2-\alpha})^{-\frac{2+\alpha}{2\alpha}}x^{-\frac{2+\alpha}{2}},\quad x\geq(\frac{2+\alpha}{2-\alpha})^{-\frac{1}{\alpha}},
\end{align*}
\begin{align*}
\mathbb{P}(X_{1}\leq x)=2^{-1}(\frac{2+\alpha}{2-\alpha})^{-\frac{2+\alpha}{2\alpha}}(-x)^{-\frac{2+\alpha}{2}},\quad x\leq-(\frac{2+\alpha}{2-\alpha})^{-\frac{1}{\alpha}}.
\end{align*}
By the definition, it is easy to verify that $m=\mathbb{E}X_{1}=0$ and $v=\mathbb{E}|X_{1}-m|^{\alpha}=1$.} We can get figures for the upper bound of $\hat{\theta}$, the {upper bound and lower bound} of ${\bar{X}}$. It is obvious from Figures \ref{figure2}-\ref{figure5} that
the $\hat{\theta}$ has a better performance when $\epsilon$ is small enough.
We can also see that the smaller the $\alpha$ is, the better the performance of
$\hat \theta$ will be comparing with that of ${\bar{X}}$. The parameters for Figures \ref{figure2}-\ref{figure5} are in Table \ref{Tab1}, where $0.001:0.001:0.08$ {means the range of $\epsilon$ is from $0.001$ to $0.08$ with step-size $0.001$. The ranges of $\epsilon$ in Table \ref{Tab1} satisfy \eqref{en} and the values of $n$ in Table \ref{Tab1} satisfy \eqref{assu} and \eqref{en}.}
\begin{table}[htbp]
\centering
\caption{ Parameters in Figures \ref{figure2}-\ref{figure5} }
\begin{tabular}{c|c|c|c}
\hline
&$\alpha$  & $\epsilon$ & n  \\
\hline
Figure \ref{figure2}
& 1.9 & 0.001:0.001:0.08  &   500 \\
 \hline
Figure \ref{figure3}
& 1.8 & 0.001:0.001:0.08  &   500 \\
 \hline
Figure \ref{figure4}
& 1.5 & 0.001:0.001:0.08  &   500 \\
 \hline
Figure \ref{figure5}
& 1.2 & 0.01:0.001:0.08  &  3000 \\
 \hline
\end{tabular}
\label{Tab1}
\end{table}

\begin{figure}
\centering
\includegraphics[height=8cm,width=14cm]{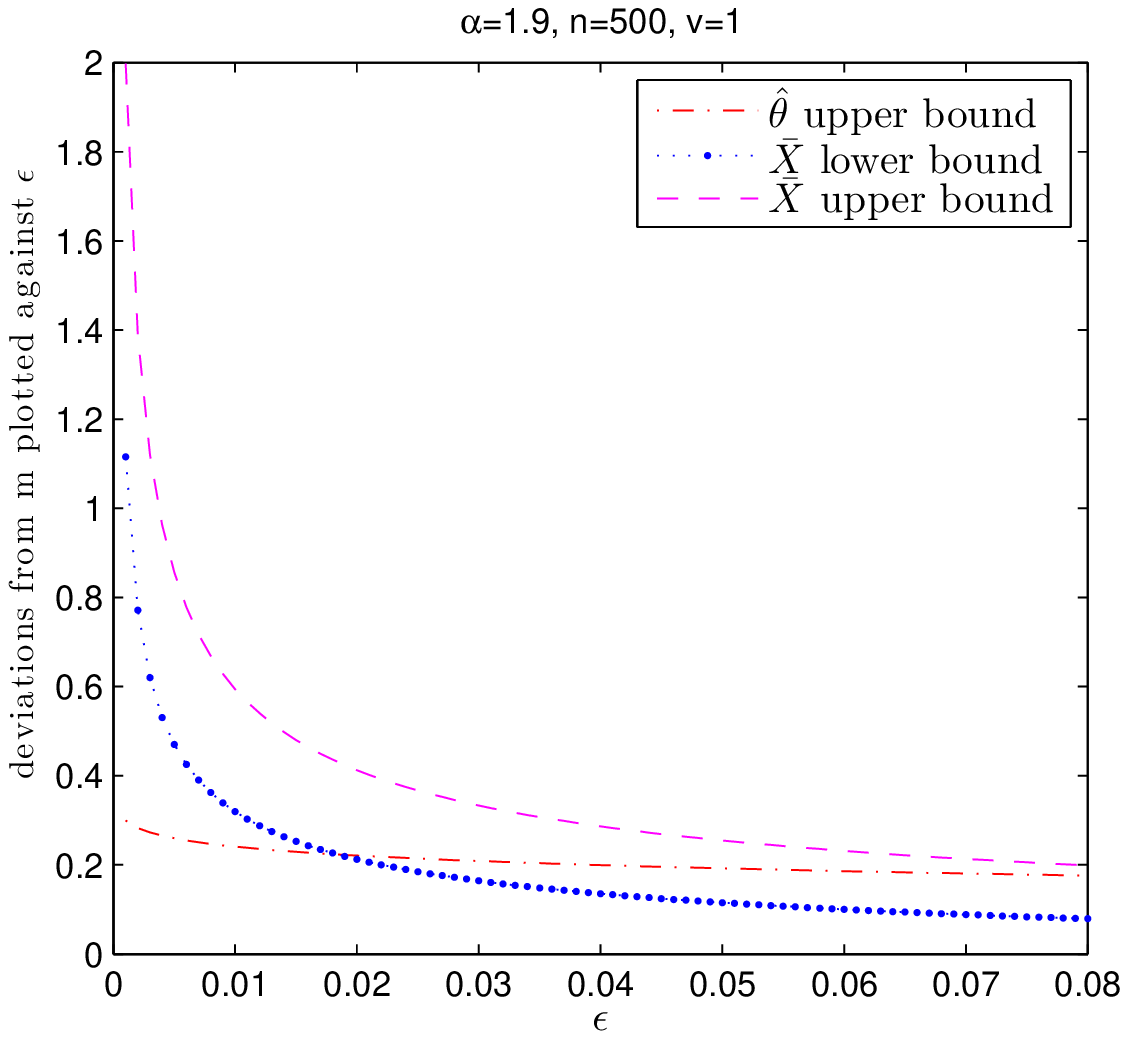}
\caption{Deviations of $\hat{\theta}$ from the sample mean, compared with those of empirical mean}
\label{figure2}
\end{figure}

\begin{figure}
\centering
\includegraphics[height=8cm,width=14cm]{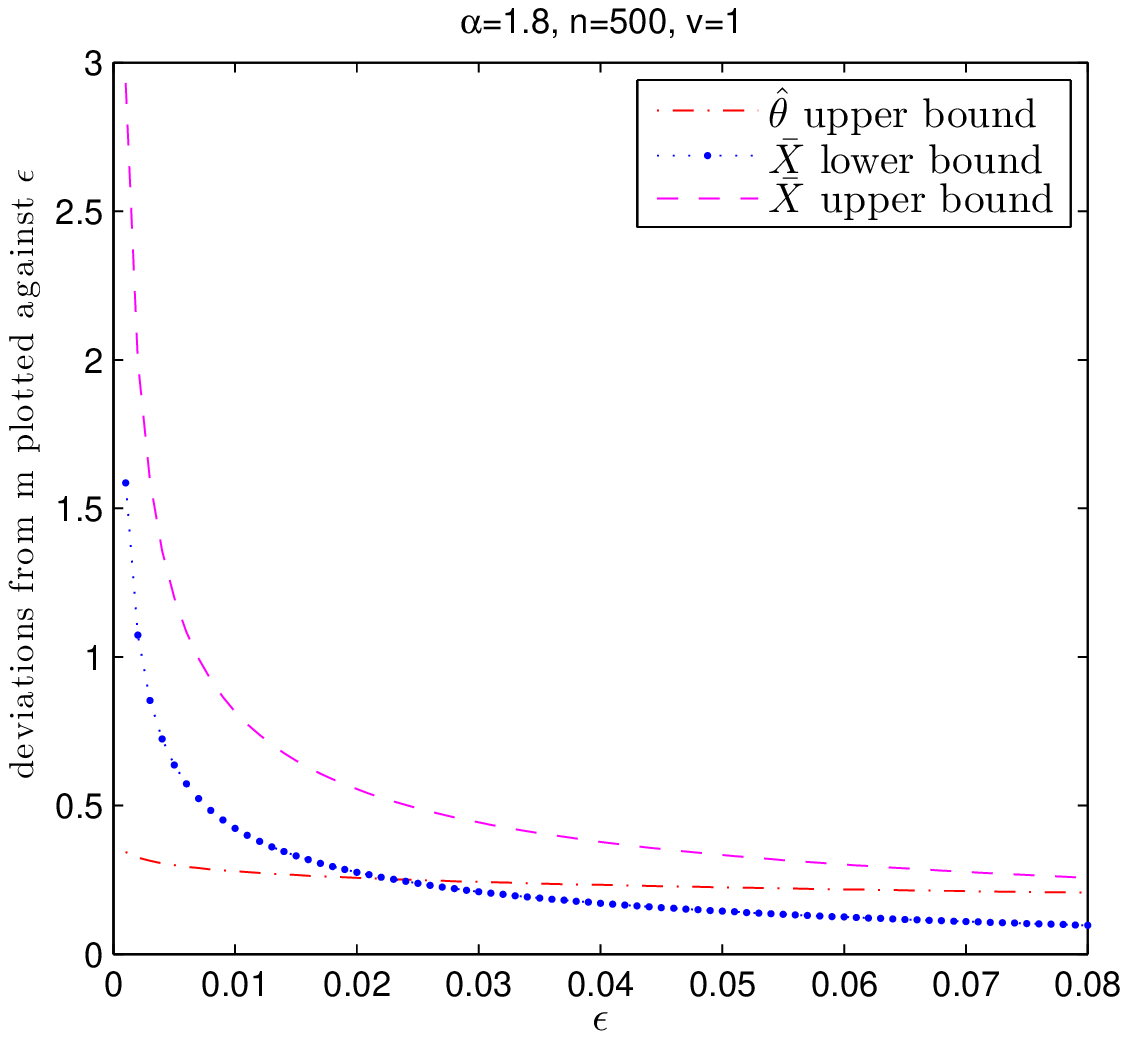}
\caption{Deviations of $\hat{\theta}$ from the sample mean, compared with those of empirical mean}
\label{figure3}
\end{figure}

\begin{figure}
\centering
\includegraphics[height=8cm,width=14cm]{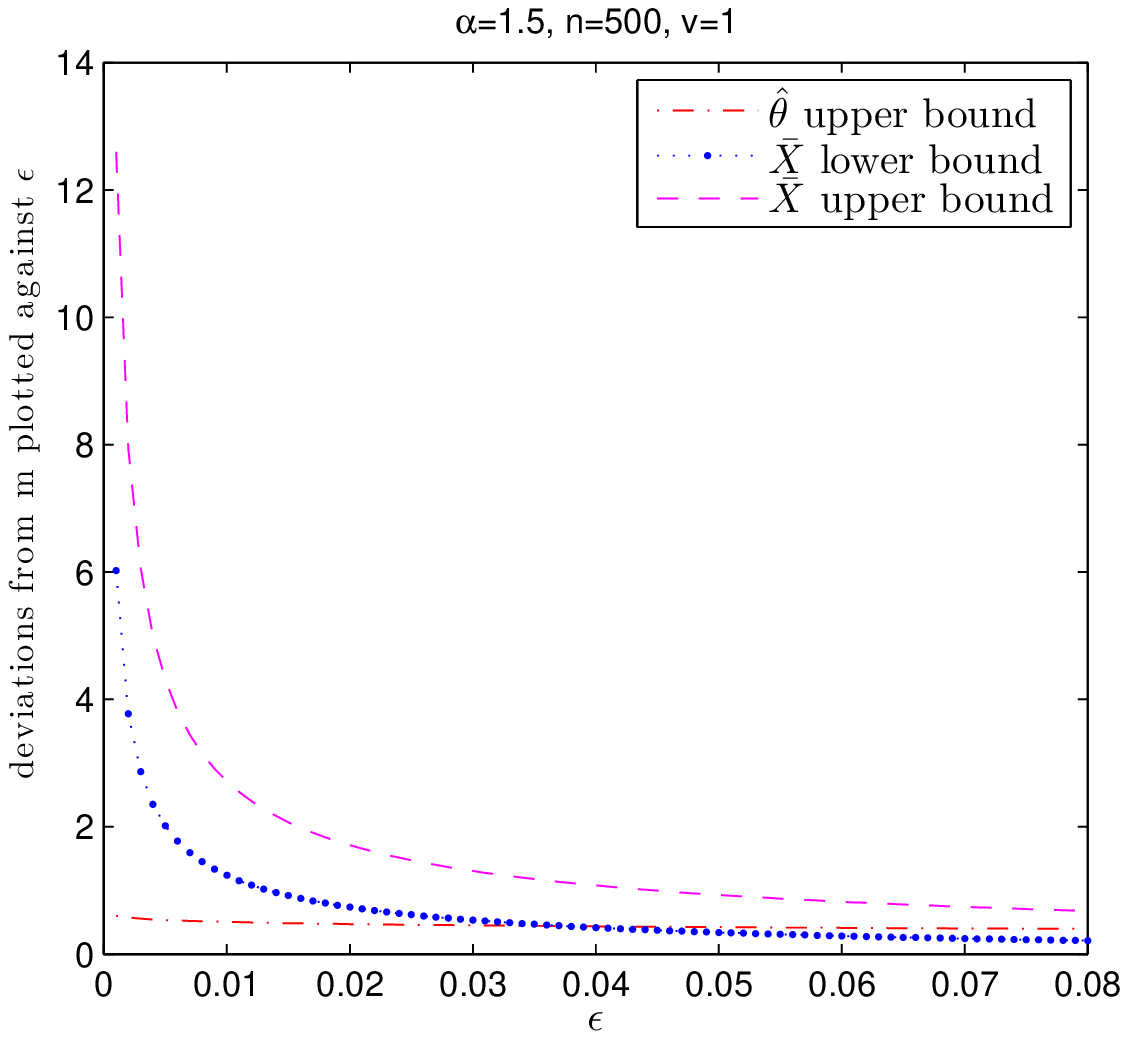}
\caption{Deviations of $\hat{\theta}$ from the sample mean, compared with those of empirical mean}
\label{figure4}
\end{figure}

\begin{figure}
\centering
\includegraphics[height=8cm,width=14cm]{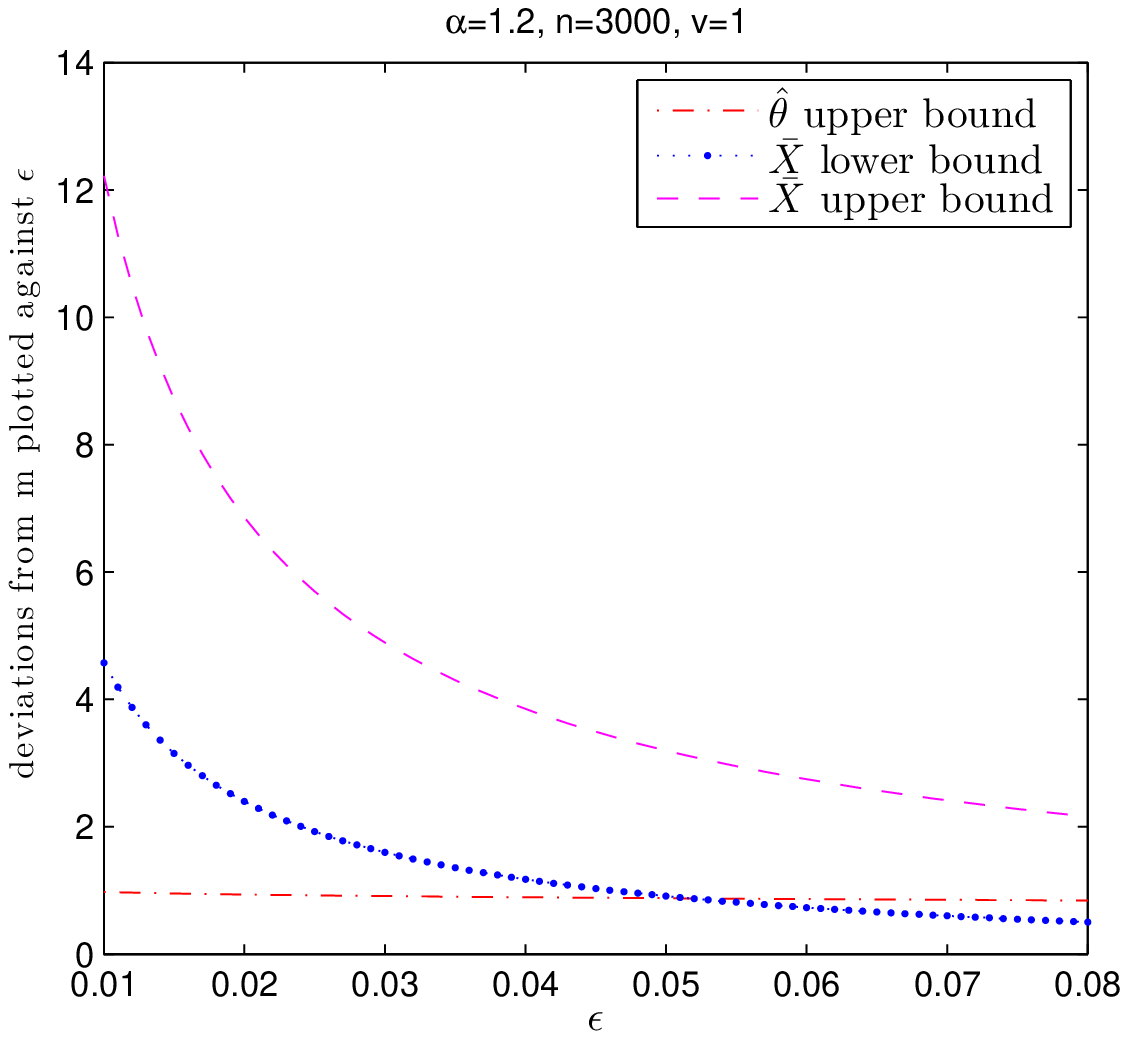}
\caption{Deviations of $\hat{\theta}$ from the sample mean, compared with those of empirical mean}
\label{figure5}
\end{figure}

\section{The deviation {upper and lower} bounds of the empirical mean estimator}\label{empirical}
At the end of the previous section, we compare our ${\rm M}$-estimator $\hat \theta$ with the empirical mean $\bar X$. In this section, we will prove the upper bounds and the corresponding lower bounds on the empirical mean used in these comparisons.

\subsection{Upper bounds}

\begin{lemma}\label{empu}
Let $(X_{i})_{i=1}^{n}$ be a sequence of random variables independently drawn from some distribution $\mathbf{\Pi}$ with mean $m$ and $\alpha$-th central moment $v$. Then, denote the empirical mean ${\bar{X}=\frac{1}{n}\sum_{i=1}^{n}X_{i},}$ we have
\begin{align*}
{
\mathbb{P}\Big(|\bar{X}-m|\geq\big(\frac{v}{\epsilon n^{\alpha-1}}\big)^{\frac{1}{\alpha}}\Big)\leq2\epsilon.}
\end{align*}
\end{lemma}

\begin{proof}
Noticing that $(X_{i}-m)_{i=1}^{n}$ are i.i.d. random variables with mean zero, by \cite[Theorem 2]{V-E}, we have
\begin{align*}
\mathbb{E}\big|\sum_{i=1}^{n}[X_{i}-m]\big|^{\alpha}\leq2\sum_{i=1}^{n}\mathbb{E}|X_{i}-m|^{\alpha}=2nv,
\end{align*}
which implies
\begin{align*}
{
\mathbb{P}\Big(|\bar{X}-m|\geq\big(\frac{v}{\epsilon n^{\alpha-1}}\big)^{\alpha}\Big)\leq\frac{\mathbb{E}|\bar{X}-m|^{\alpha}}{\frac{v}{\epsilon n^{\alpha-1}}}
\leq\frac{\frac{1}{n^{\alpha}}\mathbb{E}\big|\sum_{i=1}^{n}[X_{i}-m]\big|^{\alpha}}{\frac{v}{\epsilon n^{\alpha-1}}}\leq2\epsilon,}
\end{align*}
the desired result follows.
\end{proof}

\subsection{Lower bounds}
Comparing with Lemma \ref{empu}, the following lemma gives a lower bound for the deviations of the empirical mean for some specific distributions.

\begin{lemma}\label{emp}
For any value of the $\alpha$-th central moment $v,$ any deviation $\eta>0,$ there is some distribution $\mathbf{\Pi}$ with mean zero and $\alpha$-th {central} moment $v$ such that
\begin{align}\label{lower1}
\mathbb{P}(\bar{X}\geq\eta)=\mathbb{P}(\bar{X}\leq-\eta)\geq\frac{v}{3n^{\alpha-1}\eta^{\alpha}}\big(1-\frac{v}{n^{\alpha}\eta^{\alpha}}\big)^{n-1},
\end{align}
where $\bar{X}=\frac{1}{n}\sum_{i=1}^{n}X_{i}$ with $(X_{i})_{i=1}^{n}$ independently drawn from the distribution $\mathbf{\Pi}.$ Furthermore, if
{
\begin{align}\label{en}
\epsilon<(3e)^{-1} {\text \ \ and \ \ }
n\geq2,
\end{align}
}
the inequality
\begin{align*}
|\bar{X}-m|\geq\big(\frac{v}{3n^{\alpha-1}\epsilon}\big)^{\frac{1}{\alpha}}(1-\frac{3e\epsilon}{n})^{\frac{n-1}{\alpha}}
\end{align*}
holds with probability at least $2\epsilon.$
\end{lemma}

\begin{proof}
Let us consider the random variable $X,$ which has the following distribution: 
{
\begin{align*}
\mathbb{P}(X=0)=1-\frac{v}{n^{\alpha}\eta^{\alpha}}, \quad \mathbb{P}(X=n\eta)=\mathbb{P}(X=-n\eta)=\frac{v}{3n^{\alpha}\eta^{\alpha}}
\end{align*}
and
\begin{align*}
\mathbb{P}\big(X\in(x,\infty)\backslash\{n\eta\}\big)=\frac{q}{2\gamma}x^{-\gamma}, \quad x\in(p,\infty)\backslash\{n\eta\}
\end{align*}
\begin{align*}
\mathbb{P}\big(X\in(-\infty,x)\backslash\{-n\eta\}\big)=\frac{q}{2\gamma}|x|^{-\gamma}, \quad x\in(-\infty,-p)\backslash\{-n\eta\},
\end{align*}
where $\gamma\in(\alpha,2),$ $p=\big(\frac{\gamma-\alpha}{\gamma}\big)^{\frac{1}{\alpha}}n\eta$ and $q=\frac{\gamma v}{3}\big(\frac{\gamma-\alpha}{\gamma}\big)^{\frac{\gamma}{\alpha}}(n\eta)^{\gamma-\alpha}.$ It is easy to check that $\mathbb{E}X=0$ and
\begin{align*}
\mathbb{E}|X|^{\alpha}=(n\eta)^{\alpha}\frac{v}{3n^{\alpha}\eta^{\alpha}}+{ (n\eta)}^{\alpha}\frac{v}{3n^{\alpha}\eta^{\alpha}}+\frac{q}{\gamma-\alpha}p^{\alpha-\gamma}=\frac{v}{3}+\frac{v}{3}+\frac{v}{3}=v.
\end{align*}
}

Let $(X_{i})_{i=1}^{n}$ be i.i.d., which have the same distribution as $X.$ Then,
\begin{align*}
\mathbb{P}(\bar{X}\geq\eta)
=\mathbb{P}(\bar{X}\leq-\eta)
\geq\mathbb{P}(\bar{X}=\eta)
\geq \frac{v}{3n^{\alpha-1}\eta^{\alpha}}\big(1-\frac{v}{n^{\alpha}\eta^{\alpha}}\big)^{n-1},
\end{align*}
so (\ref{lower1}) is proved.

Take $\eta=\big(\frac{v}{3n^{\alpha-1}\epsilon}\big)^{\frac{1}{\alpha}}(1-\frac{3e\epsilon}{n})^{\frac{n-1}{\alpha}},$ we have
\begin{align*}
\frac{v}{3n^{\alpha-1}\eta^{\alpha}}\big(1-\frac{v}{n^{\alpha}\eta^{\alpha}}\big)^{n-1}=\epsilon(1-\frac{3e\epsilon}{n})^{-(n-1)}\Big(1-\frac{3\epsilon}{n(1-\frac{3e\epsilon}{n})^{n-1}}\Big)^{n-1}.
\end{align*}
If $\epsilon<(3e)^{-1},$ then $(1-\frac{3e\epsilon}{x})^{x-1}\geq(1-\frac{1}{x})^{x-1}.$ For any $x\geq1,$ we denote $f(x)=(1-\frac{1}{x})^{x-1},$ then
\begin{align*}
f'(x)=&(1-\frac{1}{x})^{x-1}\Big(\log(1-\frac{1}{x})+\frac{(x-1)}{x^{2}(1-\frac{1}{x})}\Big)\\
=&(1-\frac{1}{x})^{x-1}\Big(\log(1-\frac{1}{x})+\frac{1}{x}\Big).
\end{align*}
Noting that $(1-\frac{1}{x})^{x-1}>0,$ let $g(x)=\log(1-\frac{1}{x})+\frac{1}{x}$ for $x\geq1,$ then we have
\begin{align*}
g'(x)=\frac{1}{x^{2}(1-\frac{1}{x})}-\frac{1}{x^{2}}=\frac{1}{x^{2}(x-1)}>0
\end{align*}
and
\begin{align*}
\lim_{x\rightarrow\infty}g(x)=\lim_{x\rightarrow\infty}\big[\log(1-\frac{1}{x})+\frac{1}{x}\big]=0,
\end{align*}
which imply $g(x)\leq0,$ so we have $f'(x)\leq0$ for $x\geq1.$ What's more, we have
\begin{align*}
\lim_{x\rightarrow\infty}f(x)=\lim_{x\rightarrow\infty}(1-\frac{1}{x})^{x-1}=e^{-1},
\end{align*}
which implies $(1-\frac{3e\epsilon}{n})^{n-1}\geq e^{-1}.$ Therefore, we have
\begin{align*}
\frac{v}{3n^{\alpha-1}\eta^{\alpha}}\big(1-\frac{v}{n^{\alpha}\eta^{\alpha}}\big)^{n-1}\geq\epsilon(1-\frac{3e\epsilon}{n})^{-(n-1)}\Big(1-\frac{3e\epsilon}{n}\Big)^{n-1}=\epsilon.
\end{align*}
The proof is complete.
\end{proof}

\section{$\ell_{1}$-regression for heavy-tailed samples having finite $\alpha$-th moment with $\alpha \in (1,2)$}\label{l1regression}

 The linear regression considered in \cite{Z-Z} aims to find the unknown minimizer $\theta^{*}$ of the following minimization problem:
\begin{align}\label{min}
\min_{\mathbf{\theta}\in\mathbf{\Theta}} R_{\ell_{1}}(\theta) \ \ \ \ {\rm with} \ \ \ R_{\ell_{1}}(\theta)=\mathbb{E}_{(\mathbf{x},y)\sim{\bf\Pi}}\big[|\mathbf{x}^{T}\mathbf{\theta}-y|\big],
\end{align}
where ${\bf\Pi}$ is the population's distribution, and {$\mathbf{\Theta} \subseteq \R^d$} is the set in which $\theta^{*}$ is located. In practice, ${\bf\Pi}$ is not known, one usually draws a data set $\mathcal{T}=(\mathbf{x}_{1},y_{1}),\cdots,(\mathbf{x}_{n},y_{n})$ from $\mathbf{\Pi}$ and consider the following empirical optimization problem:
\begin{align*}
\min_{\mathbf{\theta}\in\mathbf{\Theta}}\widehat{R}_{\ell_{1}}(\mathbf{\theta}) \ \ \ {\rm with} \ \ \ \widehat{R}_{\ell_{1}}(\mathbf{\theta})=\frac{1}{n}\sum_{i=1}^{n}|\mathbf{x}_{i}^{T}\mathbf{\theta}-y_{i}|.
\end{align*}

Inspired by Catoni's work, Zhang et al. \cite{Z-Z} considered the case that ${\bf\Pi}$ is heavy tailed with finite variance and proposed a new minimization problem
\begin{align}\label{mintrun}
\min_{\mathbf{\theta}\in\mathbf{\Theta}} \widehat R_{\varphi, \ell_{1}}(\theta) \ \ \ {\rm with} \ \ \ \widehat R_{\varphi, \ell_{1}}(\theta)=\frac{1}{n\beta}\sum_{i=1}^{n}\varphi(\beta|y_{i}-\mathbf{x}_{i}^{T}\theta|),
\end{align}
where $\varphi$ is the same as that in \cite{C12} and $\beta>0$ is to be determined later.

Thanks to the analysis of Section \ref{deviation}, we extend the results in \cite{Z-Z} to the case in which samples can have finite $\alpha$-th moment with $\alpha \in (1,2)$, the approach is by replacing the original $\varphi$ with \eqref{trunc}.

\subsection{Main results of this section}

Before stating the main results, we first give some definitions and assumptions.


\begin{definition} Let $(\mathbf{\Theta},d)$ be a metric space, and $\mathbf{K}$ be a subset of $\mathbf{\Theta}$. Then a subset $\mathcal{N}\subseteq\mathbf{K}$ is called an {$\e$-net} of $\mathbf{K}$ if for every $\theta\in\mathbf{K}$, we can find a $\tilde{\theta}\in\mathcal{N}$ such that $d(\theta,\tilde{\theta})\leq\e$. The covering number is the minimal cardinality of the $\e$-net of $\mathbf{\Theta}$ and denoted by $N(\bf{\Theta},\e)$.
\end{definition}

We shall assume:

{\bf {Assumption A1}} (i) The domain $\mathbf{\Theta}$ is totally bounded, that is, for any $\e>0$, there exists a finite $\e$-net of $\mathbf{\Theta}.$

(ii) The expectation of the $\alpha$-th moment of $\mathbf{x}$ is bounded, that is,
\begin{align*}
\mathbb{E}_{(\mathbf{x,y})\sim\mathbf{\Pi}}[|\mathbf{x}|^{\alpha}]<\infty.
\end{align*}

(iii) The $\ell_{\alpha}$-risk of all $\theta\in\mathbf{\Theta}$ is bounded, that is,
\begin{align*}
\sup_{\theta\in\mathbf{\Theta}}R_{\ell_{\alpha}}(\theta)=\sup_{\theta\in\mathbf{\Theta}}\mathbb{E}_{(\mathbf{x},y)\sim\mathbf{\Pi}}[|y-\mathbf{x}^{T}\theta|^{\alpha}]<\infty.
\end{align*}

Then, we can state the second theorem, which will be proved in subsection \ref{pre}.

\begin{theorem}\label{main2}
Let $\theta^{*}$ and $\hat{\theta}$ be the minimizers of (\ref{min}) and (\ref{mintrun}), respectively. Under {\bf Assumption A1}, for any $\delta\in(0,\frac{1}{2}),$ with probability at least $1-2\delta$, we have
\begin{align*}
&R_{\ell_{1}}(\hat{\theta})-R_{\ell_{1}}(\theta^{*})\\
\leq&2\e\mathbb{E}|\mathbf{x}_{1}| +\big(\frac{2^{\alpha-1}\e^{\alpha}}{\alpha} \mathbb{E}|\mathbf{x}_{1}|^{\alpha} +\frac{2^{\alpha-1}+1}{\alpha} \sup_{\theta\in\bf{\Theta}}R_{\ell_{\alpha}}(\theta) \big)\beta^{\alpha-1}+\frac{1}{n\beta} \log\frac{N(\bf{\Theta},\e)}
{\delta^{2}}
\end{align*}
for any $\e>0.$ Furthermore, let
\begin{align*}
\beta=\big(\frac{1}{n}\log\frac{N(\mathbf{\Theta},\e)}{\delta^{2}}\big)^{\frac{1}{\alpha}},
\end{align*}
we have
\begin{align}\label{end}
&R_{\ell_{1}}(\hat{\theta})-R_{\ell_{1}}(\theta^{*})\nonumber\\
\leq&2\e\mathbb{E}|\mathbf{x}_{1}| +\big(\frac{2^{\alpha-1}\e^{\alpha}}{\alpha} \mathbb{E}|\mathbf{x}_{1}|^{\alpha}+\frac{2^{\alpha-1}+1}{\alpha}\sup_{\theta\in\bf{\Theta}}R_{\ell_{\alpha}}(\theta)+1\big)
\big(\frac{1}{n}\log\frac{N({\bf \Theta},\e)}{\delta^{2}}\big)^{\frac{\alpha-1}{\alpha}}.
\end{align}
\end{theorem}

{ In order to compute the covering number, we further assume:}

{\bf Assumption A2} The domain ${\bf{\Theta}}\subseteq\mathbb{R}^{d}$, and its radius is bounded by a constant $r$, that is,  
\begin{align*}
|\theta|\leq r ,\quad \forall\theta\in{\bf{\Theta}}.
\end{align*}

Then, we have the following corollary, which will be proved in subsection \ref{pre}.

\begin{corollary}\label{explicit}
Keep the same notations and assumptions in Theorem \ref{main2}. In addition, we suppose the {\bf Assumption A2} holds. Then, for any $\delta\in(0,\frac{1}{2}),$ with probability at least $1-2\delta$, we have
\begin{align*}
&R_{\ell_{1}}(\hat{\theta})-R_{\ell_{1}}(\theta^{*})\\
\leq&\frac{2}{n}\mathbb{E}|\mathbf{x}_{1}|+\big(\frac{2^{\alpha-1}}{\alpha n^{\alpha}} \mathbb{E}|\mathbf{x}_{1}|^{\alpha}+\frac{2^{\alpha-1}+1}{\alpha}\sup_{\theta\in\bf{\Theta}}R_{\ell_{\alpha}}(\theta)+1\big)
\Big(\frac{1}{n}\big(d\log(6nr)+\log\frac{1}{\delta^{2}}\big)\Big)^{\frac{\alpha-1}{\alpha}}\\
=&O\Big(\big(\frac{d\log n}{n}\big)^{\frac{\alpha-1}{\alpha}}\Big).
\end{align*}
\end{corollary}

\subsection{Proof of Theorem \ref{main2} and Corollary \ref{explicit}}\label{pre}

Before proving the Theorem \ref{main2}, we first give the following auxiliary lemmas.

\begin{lemma}\label{first}
Keep the same notations and assumptions as in Theorem \ref{main2}. Then, the following inequality
\begin{align*}
\widehat{R}_{\varphi,\ell_{1}}(\theta^{*})-R_{\ell_{1}}(\theta^{*})\leq\frac{\beta^{\alpha-1}}{\alpha}R_{\ell_{\alpha}}(\theta^{*})+\frac{1}{n \beta}\log\frac{1}{\delta}
\end{align*}
holds with probability at least $1-\delta$.
\end{lemma}

\begin{proof}
Noticing that $(\mathbf{x}_{i},y_{i}),$ $i=1,\cdots,n,$ are i.i.d., by (\ref{trunc}), we have
\begin{align*}
\mathbb{E}\big[\exp\big(n\beta \widehat{R}_{\varphi,\ell_{1}}(\theta^{*})\big)\big]=&\mathbb{E}\big[\exp\big(\sum_{i=1}^{n}\varphi(\beta|y_{i}-{\bf x}_{i}^{T}\theta^{*}|)\big)\big]\\
=&\Big[\mathbb{E}\big[\exp\big(\varphi(\beta|y_{1}-{\bf x}_{1}^{T}\theta^{*}|)\big)\big]\Big]^{n}\\
\leq&\Big[\mathbb{E}\big[1+\beta|y_{1}-{\bf x}_{1}^{T}\theta^{*}|+\frac{\beta^{\alpha}|y_{1}-{\bf x}_{1}^{T}\theta^{*}|^{\alpha}}{\alpha}\big]\Big]^{n},
\end{align*}
then, by the inequality $1+x\leq e^{x}$ for all $x\in \R$, we have
\begin{align*}
\mathbb{E}\big[\exp\big(n\beta \widehat{R}_{\varphi,\ell_{1}}(\theta^{*})\big)\big]
\leq&\big[1+\beta R_{\ell_{1}}(\theta^{*})+\frac{\beta^{\alpha}}{\alpha}R_{\ell_{\alpha}}(\theta^{*})\big] ^{n}\\
\leq& \exp\big(n\beta R_{\ell_{1}}(\theta^{*})+\frac{n\beta^{\alpha}}{\alpha}R_{\ell_{\alpha}}(\theta^{*})\big).
\end{align*}
Therefore, by Markov inequality, we have
\begin{align*}
&\mathbb{P}\Big(n\beta \widehat{R}_{\varphi,\ell_{1}}(\theta^{*})\geq n\beta R_{\ell_{1}}(\hat{\theta})+\frac{n\beta^{\alpha}}{\alpha}R_{\ell_{\alpha}}(\theta^{*})+\log\frac{1}{\delta}\Big)\\
=&\mathbb{P}\Big(\exp\big(n\beta \widehat{R}_{\varphi,\ell_{1}}(\theta^{*})\big)\geq \exp\big(n\beta R_{\ell_{1}}(\theta^{*})+\frac{n\beta^{\alpha}}{\alpha}R_{\ell_{\alpha}}(\theta^{*})+\log\frac{1}{\delta}\big)\Big)\\
\leq&\frac{\mathbb{E}\big[\exp\big(n\beta \widehat{R}_{\varphi,\ell_{1}}(\theta^{*})\big)\big]}{\exp\big(n\beta R_{\ell_{1}}(\theta^{*})+\frac{n\beta^{\alpha}}{\alpha}R_{\ell_{\alpha}}(\theta^{*})+\log\frac{1}{\delta}\big)}\leq\delta.
\end{align*}
The proof is complete.
\end{proof}

\begin{lemma}\label{pre2}
For any $\e>0,$ {let $\mathcal{N}({\bf\Theta},\e)$ be an $\e$-net of ${\bf\Theta}$ with cardinality $N({\bf\Theta},\e).$} Then, with probability at least $1-\delta$, the following inequality
\begin{align*}
&-\frac{1}{n \beta}\sum_{i=1}^{n}\varphi\big(\beta|y_{i}-\mathbf{x}_{i}^{T}\tilde{\theta}|-\beta\e|\mathbf{x}_{i}|\big)\\
&\leq-R_{\ell_{1}}(\tilde{\theta})+\e\mathbb{E}|\mathbf{x}_{1}|+\frac{(2\beta)^{\alpha-1}}{\alpha}\sup_{\theta\in{\bf\Theta}}R_{\ell_{\alpha}}(\theta)
+\frac{(2\beta)^{\alpha-1}\e^{\alpha}}{\alpha}\mathbb{E}|\mathbf{x}_{1}|^{\alpha}+\frac{1}{n\beta}\log\frac{N({\bf\Theta},\e)}{\delta}
\end{align*}
holds for all $\tilde{\theta}\in\mathcal{N}({\bf\Theta},\e).$ 
\end{lemma}

\begin{proof}
For a fixed $\tilde{\theta}\in\mathcal{N}({\bf\Theta},\e),$ noticing that $(\mathbf{x}_{i},y_{i}),$ $i=1,\cdots,n,$ are i.i.d., by (\ref{trunc}), we have
\begin{align*}
&\mathbb{E}\big[\exp\big(-\sum_{i=1}^{n} \varphi(\beta|y_{i}-\mathbf{x}_{i}^{T} \tilde{\theta}|-\beta\e|\mathbf{x}_{i}|)\big)\big]\\
=&\Big[\mathbb{E}\big[\exp\big(-\varphi (\beta|y_{1}-\mathbf{x}_{1}^{T} \tilde{\theta}|-\beta\e|\mathbf{x}_{1}|)\big)\big]\Big]^{n}\\
\leq&\Big[\mathbb{E}\big[1-\beta|y_{1}- \mathbf{x}_{1}^{T}\tilde{\theta}| +\beta\e|\mathbf{x}_{1}|+\frac{\beta^{\alpha} (|y_{1}-\mathbf{x}_{1}^{T} \tilde{\theta}|-\e|\mathbf{x}_{1}|)^{\alpha}}{\alpha} \big]\Big]^{n}\\
=&\Big[1-\beta R_{\ell_{1}}(\tilde{\theta})+\beta\e \mathbb{E}|\mathbf{x}_{1}| +\frac{\beta^{\alpha}}{\alpha} \mathbb{E}\big[\big||y_{1}-\mathbf{x}_{1}^{T} \tilde{\theta}|-\e|\mathbf{x}_{1}|\big |^{\alpha}\big]\Big]^{n},
\end{align*}
then, by (\ref{CR}) and the inequality $1+x\leq e^{x}$ for all $x\in \R$, we have
\begin{align*}
&\mathbb{E}\big[\exp\big(-\sum_{i=1}^{n} \varphi(\beta|y_{i}-\mathbf{x}_{i}^{T} \tilde{\theta}|-\beta\e|\mathbf{x}_{i}|)\big)\big]\\
\leq&\Big[1-\beta R_{\ell_{1}}(\tilde{\theta})+\beta\e\mathbb{E} |\mathbf{x}_{1}|+\frac{\beta^{\alpha}2^{\alpha-1}}{\alpha}R_{\ell_{\alpha}}(\tilde{\theta})
+\frac{\beta^{\alpha}\e^{\alpha}2^{\alpha-1}}{\alpha}\mathbb{E}|\mathbf{x}_{1}|^{\alpha}]\Big]^{n}\\
\leq&\exp\Big[n\big(-\beta R_{\ell_{1}}(\tilde{\theta})+\beta\e\mathbb{E}|\mathbf{x}_{1}|+\frac{\beta^{\alpha}2^{\alpha-1}}{\alpha}R_{\ell_{\alpha}}(\tilde{\theta})
+\frac{\beta^{\alpha}\e^{\alpha}2^{\alpha-1}}{\alpha}\mathbb{E}|\mathbf{x}_{1}|^{\alpha}]\big)\Big].
\end{align*}
Therefore, by Markov inequality, we have
\begin{align*}
&\mathbb{P}\Big(-\sum_{i=1}^{n} \varphi(\beta|y_{i}-\mathbf{x}_{i}^{T} \tilde{\theta}|-\beta\e|\mathbf{x}_{i}|)\\
&\qquad\geq n\big(-\beta R_{\ell_{1}}(\tilde{\theta})+\beta\e\mathbb{E} |\mathbf{x}_{1}|+\frac{\beta^{\alpha}2^{\alpha-1}} {\alpha}R_{\ell_{\alpha}}(\tilde{\theta})
+\frac{\beta^{\alpha}\e^{\alpha}2^{\alpha-1}}{\alpha} \mathbb{E}|\mathbf{x}_{1}|^{\alpha}]\big) +\log\frac{1}{\delta'}\Big)\\
\leq&\frac{\mathbb{E}\big[\exp\big(-\sum_{i=1}^{n} \varphi(\beta|y_{i}-\mathbf{x}_{i}^{T}\tilde{\theta}| -\beta\e|\mathbf{x}_{i}|)\big)\big]} {\exp\Big[n\big(-\beta R_{\ell_{1}}(\tilde{\theta})+\beta\e\mathbb{E} |\mathbf{x}_{1}|+\frac{\beta^{\alpha}2^{\alpha-1}} {\alpha}R_{\ell_{\alpha}}(\tilde{\theta})
+\frac{\beta^{\alpha}\e^{\alpha}2^{\alpha-1}}{\alpha} \mathbb{E}|\mathbf{x}_{1}|^{\alpha}]\big) +\log\frac{1}{\delta'}\Big]}\leq\delta',
\end{align*}
where $\delta'\in(0,1)$, which will be chosen later. Hence, for a fixed $\tilde{\theta}\in\mathcal{N}({\bf\Theta},\e),$ with probability at most $\delta',$ we have
\begin{align*}
&-\frac{1}{n\beta}\sum_{i=1}^{n}\varphi(\beta|y_{i}-\mathbf{x}_{i}^{T}\tilde{\theta}|-\beta\e|\mathbf{x}_{i}|)\\
\geq&-R_{\ell_{1}}(\tilde{\theta})+\e\mathbb{E}|\mathbf{x}_{1}|+\frac{(2\beta)^{\alpha-1}}{\alpha}R_{\ell_{\alpha}}(\tilde{\theta})
+\frac{(2\beta)^{\alpha-1}\e^{\alpha}}{\alpha}\mathbb{E}|\mathbf{x}_{1}|^{\alpha}]+\frac{1}{n\beta}\log\frac{1}{\delta'}.
\end{align*}
Therefore, since the set $\mathcal{N}({\bf\Theta},\e)$ has $N({\bf\Theta},\e)$ elements, we have
\begin{align*}
&\mathbb{P}\Big(\bigcap_{\tilde{\theta}\in \mathcal{N}({\bf\Theta},\e)}\big\{-\frac{1}{n\beta}\sum_{i=1}^{n} \varphi(\beta|y_{i}-\mathbf{x}_{i}^{T}\tilde{\theta}| -\beta\e|\mathbf{x}_{i}|)\leq\\
&\qquad\qquad\quad-R_{\ell_{1}}(\tilde{\theta})+\e\mathbb{E} |\mathbf{x}_{1}|+\frac{(2\beta)^{\alpha-1}} {\alpha}R_{\ell_{\alpha}}(\tilde{\theta})
+\frac{(2\beta)^{\alpha-1}\e^{\alpha}}{\alpha} \mathbb{E}|\mathbf{x}_{1}|^{\alpha}]\big)+\frac{1}{n\beta}\log\frac{1}{\delta'}\big\}\Big)\\
\geq&1-N({\bf\Theta},\e)\delta'.
\end{align*}
Finally, taking $\delta'=\frac{\delta}{N({\bf\Theta},\e)},$ with probability at least $1-\delta,$ the following inequality
\begin{align*}
&-\frac{1}{n\beta}\sum_{i=1}^{n} \varphi(\beta|y_{i}-\mathbf{x}_{i}^{T}\tilde{\theta}| -\beta\e|\mathbf{x}_{i}|)\\
\leq&-R_{\ell_{1}}(\tilde{\theta}) +\e\mathbb{E}|\mathbf{x}_{1}| +\frac{(2\beta)^{\alpha-1}}{\alpha}R_{\ell_{\alpha}} (\tilde{\theta})
+\frac{(2\beta)^{\alpha-1}\e^{\alpha}}{\alpha} \mathbb{E}|\mathbf{x}_{1}|^{\alpha}] +\frac{1}{n\beta}\log\frac{N({\bf\Theta},\e)} {\delta}\\
\leq&-R_{\ell_{1}}(\tilde{\theta})+\e\mathbb{E} |\mathbf{x}_{1}|+\frac{(2\beta)^{\alpha-1}}{\alpha} \sup_{\theta\in{\bf\Theta}}R_{\ell_{\alpha}}(\theta)
+\frac{(2\beta)^{\alpha-1}\e^{\alpha}}{\alpha} \mathbb{E}|\mathbf{x}_{1}|^{\alpha} +\frac{1}{n\beta}\log\frac{N({\bf\Theta},\e)}{\delta}
\end{align*}
holds for all $\tilde{\theta}\in\mathcal{N}({\bf\Theta},\e).$ The proof is complete.
\end{proof}

Based on Lemma \ref{pre2}, we have the following lemma.

\begin{lemma}\label{second}
Keep the same notations and assumptions as in Theorem \ref{main2}. Then, for any $\e>0,$ the following inequality
\begin{align*}
&R_{\ell_{1}}(\hat{\theta})-\widehat{R}_{\varphi,\ell_{1}}(\hat{\theta})\\
\leq&2\e\mathbb{E}|\mathbf{x}_{1}|+\frac{(2\beta)^{\alpha-1}}{\alpha}\sup _{\theta\in{\bf \Theta}}R_{\ell_{\alpha}}(\theta)+\frac{(2\beta)^{\alpha-1}\e^{\alpha}}{\alpha} \mathbb{E}|\mathbf{x}_{1}|^{\alpha}+\frac{1}{n\beta} \log \frac{N({\bf\Theta},\e)}{\delta}
\end{align*}
holds with probability at least $1-\delta.$
\end{lemma}

\begin{proof}
Since $\hat{\theta}\in\mathbf{\Theta},$ there exists a $\tilde{\theta}\in\mathcal{N}({\bf\Theta},\e)$ such that
\begin{align*}
|\hat{\theta}-\tilde{\theta}|\leq\e,
\end{align*}
which implies
\begin{align}\label{l1}
|y_{i}-\mathbf{x}_{i}^{T}\hat{\theta}|\geq|y_{i}-\mathbf{x}_{i}^{T}\tilde{\theta}|-|\mathbf{x}_{i}^{T}(\tilde{\theta}-\hat{\theta})|\geq|y_{i}-\mathbf{x}_{i}^{T} \tilde{\theta}|-\e|\mathbf{x}_{i}|.
\end{align}
Then, since $\varphi(\cdot)$ is non-decreasing, we have
\begin{align*}
\widehat{R}_{\varphi,\ell_{1}}(\hat{\theta})=\frac{1}{n \beta}\sum_{i=1}^{n}\varphi\big(\beta|y_{i}-\mathbf{x}_{i}^{T}\hat{\theta}|\big)\ge\frac{1}{n\beta}\sum_{i=1}^{n} \varphi\big(\beta|y_{i}-\mathbf{x}_{i}^{T}\tilde{\theta}|-\beta\e|\mathbf{x}_{i}|\big),
\end{align*}
by Lemma \ref{pre2}, with probability at least $1-\delta$, we have
\begin{align*}
&\widehat{R}_{\varphi,\ell_{1}}(\hat{\theta})\\
	\ge&R_{\ell_{1}}(\tilde{\theta})-\big[\e\mathbb{E}|\mathbf{x}_{1}|+\frac{(2\beta)^{\alpha-1}}{\alpha} \sup _{\theta \in{\bf\Theta}} R_{\ell_{\alpha}}(\theta)+\frac{(2\beta)^{\alpha-1}\e^{\alpha}}{\alpha} \mathbb{E}|\mathbf{x}_{1}|^{\alpha}+\frac{1}{n\beta} \log \frac{N({\bf\Theta},\e)}{\delta}\big].
\end{align*}
What's more, by (\ref{l1}) and triangle inequality, we have
\begin{align*}
R_{\ell_{1}}(\hat{\theta})-R_{\ell_{1}}(\tilde{\theta})=&\mathbb{E}\big[|\mathbf{x}_{1}^{T}\hat{\theta}-y_{1}|-|\mathbf{x}_{1}^{T}\tilde{\theta}-y_{1}|\big]
\leq\mathbb{E}\big[|\mathbf{x}_{1}^{T}\hat{\theta}-\mathbf{x}_{1}^{T}\tilde{\theta}|\big]\leq\e\mathbb{E}|\mathbf{x}_{1}|,
\end{align*}
which further implies that with probability at least $1-\delta$, the inequality
\begin{align*}
&\widehat{R}_{\varphi,\ell_{1}}(\hat{\theta})\\
	\ge&R_{\ell_{1}}(\hat{\theta})-\big[2\e\mathbb{E}|\mathbf{x}_{1}|+\frac{(2\beta)^{\alpha-1}}{\alpha} \sup _{\theta \in{\bf\Theta}} R_{\ell_{\alpha}}(\theta)+\frac{(2\beta)^{\alpha-1}\e^{\alpha}}{\alpha} \mathbb{E}|\mathbf{x}_{1}|^{\alpha}+\frac{1}{n\beta} \log \frac{N({\bf\Theta},\e)}{\delta}\big]
\end{align*}
{holds.} The proof is complete.
\end{proof}

Now, we can give the proof of Theorem \ref{main2}.

{\bf Proof of Theorem \ref{main2}.}
Recall
\begin{align*}
\widehat{R}_{\varphi,\ell_{1}}(\theta)=\frac{1}{n\beta}\sum_{i=1}^{n}\varphi\big(\beta|y_{i}-\mathbf{x}_{i}^{T}\theta|\big),
\end{align*}
since $\hat{\theta}$ is the minimizer of (\ref{mintrun}), we have
\begin{align*}
\widehat{R}_{\varphi,\ell_{1}}(\hat{\theta})-\widehat{R}_{\varphi,\ell_{1}}(\theta^{*})\leq0,
\end{align*}
which implies
\begin{align*}
&R_{\ell_{1}}(\hat{\theta})-R_{\ell_{1}}(\theta^{*})\\
=&\big(R_{\ell_{1}}(\hat{\theta})-\widehat{R}_{\varphi,\ell_{1}}(\hat{\theta})\big)
+\big(\widehat{R}_{\varphi,\ell_{1}}(\hat{\theta})-\widehat{R}_{\varphi,\ell_{1}}(\theta^{*})\big)+\big(\widehat{R}_{\varphi,\ell_{1}}(\theta^{*})-R_{\ell_{1}}(\theta^{*})\big)\\
\leq&\big(R_{\ell_{1}}(\hat{\theta})-\widehat{R}_{\varphi,\ell_{1}}(\hat{\theta})\big)+\big(\widehat{R}_{\varphi,\ell_{1}}(\theta^{*})-R_{\ell_{1}}(\theta^{*})\big).
\end{align*}
By Lemma \ref{second} and Lemma \ref{first}, we immediately obtain the desired result.
\qed

Now we are at the position to give the proof of Corollary \ref{explicit}.

{\bf Proof of Corollary \ref{explicit}.}  For any $\e\in(0,1],$ by \cite[Corollary 4.2.13]{Vershynin} we have
\begin{align*}
N(B_{1},\e)\leq\big(1+\frac{2}{\e}\big)^{d}\leq(\frac{3}{\e})^{d},
\end{align*}
{where $B_1=\{x\in \R^d: |x|\leq 1\}$. } Since ${\bf\Theta}\subseteq B_{r}$, we have
\begin{align}\label{number}
\log N({\bf\Theta},\e)\leq\log N(B_{r},\frac{\e}{2})\leq d\log\frac{6r}{\e}.
\end{align}
Therefore, by (\ref{end}) with $\e=\frac{1}{n},$ we have
\begin{align*}
&R_{\ell_{1}}(\hat{\theta})-R_{\ell_{1}}(\theta^{*})\\
\leq&2\e\mathbb{E}|\mathbf{x}_{1}|+\big(\frac{2^{\alpha-1}\e^{\alpha}}{\alpha} \mathbb{E}|\mathbf{x}_{1}|^{\alpha}+\frac{2^{\alpha-1}+1}{\alpha}\sup_{\theta\in\bf{\Theta}}R_{\ell_{\alpha}}(\theta)+1\big)
\big(\frac{1}{n}(d\log\frac{6r}{\e}+\log\frac{1}{\delta^{2}})\big)^{\frac{\alpha-1}{\alpha}}\\
=&\frac{2}{n}\mathbb{E}|\mathbf{x}_{1}|+\big(\frac{2^{\alpha-1}}{\alpha n^{\alpha}} \mathbb{E}|\mathbf{x}_{1}|^{\alpha}+\frac{2^{\alpha-1}+1}{\alpha}\sup_{\theta\in\bf{\Theta}}R_{\ell_{\alpha}}(\theta)+1\big)
\Big(\frac{1}{n}\big(d\log(6nr)+\log\frac{1}{\delta^{2}}\big)\Big)^{\frac{\alpha-1}{\alpha}}.
\end{align*}
The proof is complete.
\qed

{\bf Acknowledgement.} {LX is supported in part by Macao S.A.R grant FDCT  0090/2019/A2 and University of Macau grant  MYRG2018-00133-FST.}

\bibliographystyle{amsplain}

\end{document}